\documentclass{article}%
\usepackage{graphicx}
\usepackage{amsmath}
\usepackage{amsfonts}
\usepackage{amssymb}%
\providecommand{\U}[1]{\protect\rule{.1in}{.1in}}
\newtheorem{theorem}{Theorem}
{}

\newtheorem{lemma}{Lemma}
{}

\newtheorem{proposition}{Proposition}

\newenvironment{proof}[1][Proof]{\textbf{#1.} }{\ \rule{0.5em}{0.5em}}

\begin{document}

\title{On the Non-Self-adjoint Sturm-Liouville Operators in the Space of Vector-Functions}
\author{Fulya \c{S}eref and O. A. Veliev\\{\small Department of Mathematics, Dogus University, Kadik\"{o}y, \ 34722,}\\{\small \ Istanbul, Turkey.}\ {\small E-mail: fseref@dogus.edu.tr;
oveliev@dogus.edu.tr}}
\date{}
\maketitle

\begin{abstract}
In this article we obtain asymptotic formulas for the eigenvalues and
eigenfunctions of the non-self-adjoint operator generated in $L_{2}^{m}\left[
0,1\right]  $ by the Sturm-Liouville equation with $m\times m$ matrix
potential and the boundary conditions whose scalar case ($m=1$) are strongly
regular. Using these asymptotic formulas, we find a condition on the potential
for which the root functions of this operator form a Riesz basis.



\end{abstract}

\section{Introduction and Preliminary Facts}

We consider the differential operator $L_{m}(Q)$ generated in the space
$L_{2}^{m}\left[  0,1\right]  $ by the differential expression
\begin{equation}
-\mathbf{y}^{^{\prime\prime}}(x)+Q\left(  x\right)  \mathbf{y}(x) \tag{1}%
\end{equation}
and the boundary conditions whose scalar case (the case $m=1$) are strongly
regular, where $\mathbf{y}\left(  x\right)  =\left(  y_{1}\left(  x\right)
,y_{2}\left(  x\right)  ,...,y_{m}\left(  x\right)  \right)  ^{T},$ $L_{2}%
^{m}\left[  0,1\right]  $ is the set of vector-functions $\mathbf{f}\left(
x\right)  =\left(  f_{1}\left(  x\right)  ,f_{2}\left(  x\right)
,...,f_{m}\left(  x\right)  \right)  $ with $f_{k}\in L_{2}\left[  0,1\right]
$ for $k=1,2,...,m$ and $Q(x)=\left(  b_{i,j}\left(  x\right)  \right)  $ is a
$m\times m$ matrix with the complex-valued square integrable entries
$b_{i,j}.$ The norm $\left\Vert .\right\Vert $ and inner product $(.,.)$ in
$L_{2}^{m}\left[  0,1\right]  $ are defined by%
\[
\left\Vert f\right\Vert =\left(  \int\limits_{0}^{1}\left\vert f\left(
x\right)  \right\vert ^{2}dx\right)  ^{\frac{1}{2}},\text{ }(f,g)=\int
\limits_{0}^{1}\left\langle f\left(  x\right)  ,g\left(  x\right)
\right\rangle dx,
\]
where $\left\vert .\right\vert $ and $\left\langle .,.\right\rangle $ are
respectively the norm and the inner product in $\mathbb{C}^{m}.$\bigskip

A. A. Shkalikov [8-12] proved that the root functions (eigenfunctions and
associated eigenfunctions) of the operators generated by an ordinary
differential expression with summable matrix coefficients and regular boundary
conditions form a Riesz basis with parenthesis and in parenthesis should be
included only the functions corresponding to splitting eigenvalues. L. M.
Luzhina [3] generalized this result for the boundary value problems when the
coefficients depend on the spectral parameter.

In [13], the differential operator $T_{t}(Q)$ generated in the space
$L_{2}^{m}\left[  0,1\right]  $ by the differential expression (1) and the
quasiperiodic conditions%
\[
\mathbf{y}^{^{\prime}}\left(  1\right)  =e^{it}\mathbf{y}^{^{\prime}}\left(
0\right)  ,\text{ }\mathbf{y}\left(  1\right)  =e^{it}\mathbf{y}\left(
0\right)
\]
for $t\in(0,2\pi)$ and $t\neq\pi$ was considered . In [13], we proved that the
eigenvalues $\lambda_{k,j}$ of $T_{t}(Q)$ lie in the $O\left(  \frac{\ln k}%
{k}\right)  $ neighborhoods of the eigenvalues of the operator $T_{t}(C),$
where%
\begin{equation}
C=%
{\textstyle\int\nolimits_{0}^{1}}
Q\left(  x\right)  dx. \tag{2}%
\end{equation}
Note that, to obtain the asymptotic formulas of order $O(\frac{1}{k})$ for the
eigenvalues $\lambda_{k,j}$ of the differential operators generated by (1),
using the classical asymptotic expansions for the solutions of the matrix
equation $-Y^{^{\prime\prime}}+Q\left(  x\right)  Y=\lambda Y,$ it is required
\ that $\ Q$ be differentiable (see [1, 4, 5, 7]). The suggested method in
[13] gives the possibility of obtaining the asymptotic formulas of order
$O(k^{-1}\ln|k|)$ for the eigenvalues $\lambda_{k,j}$ and the normalized
eigenfunctions $\Psi_{k,j}(x)$ of $T_{t}(Q)$ when there is not any condition
about smoothness of the entries $b_{i,j}$ of $Q$. Besides, in [13], using the
asymptotic formulas, it was proved that if the eigenvalues of the matrix $C$
are simple, then the root functions of the operator $T_{t}(Q)$ form a Riesz
basis. Then, in papers [14-17], using the method of [13], we considered the
spectrum and basis property of the root functions of differential operators
generated in $L_{2}^{m}\left[  0,1\right]  $ by the differential expression of
arbitrary order and by the $t-$periodic, periodic, antiperiodic boundary
conditions and applied these investigations to the differential operators with
periodic matrix coefficients.

In this paper, we investigate the operator $L_{m}(Q)$ generated in $L_{2}%
^{m}\left[  0,1\right]  $ by (1) and the boundary conditions%
\begin{equation}
U_{i}(\mathbf{y})=\alpha_{i}\mathbf{y}^{(k_{i})}(0)+\alpha_{i,0}%
\mathbf{y}(0)+\beta_{i}\mathbf{y}^{(k_{i})}(1)+\beta_{i,0}\mathbf{y}%
(1)=0,\text{ \ }i=1,2 \tag{3}%
\end{equation}
whose scalar case (the case $m=1$)%
\begin{equation}
U_{i}(y)=\alpha_{i}y^{(k_{i})}(0)+\alpha_{i,0}y(0)+\beta_{i}y^{(k_{i}%
)}(1)+\beta_{i,0}y(1)=0,\text{ \ }i=1,2 \tag{4}%
\end{equation}
are strongly regular, where $0\leq k_{2}\leq k_{1}\leq1,$ $\alpha_{i}%
,\alpha_{i,0},\beta_{i},\beta_{i,0}$ are complex numbers and for each value of
the index $i$ at least one of the numbers $\alpha_{i},\beta_{i}$ is nonzero.

Recall that the conditions (4) are called regular if the numbers $\theta_{-1}$
and $\theta_{1}$ defined by the identity
\begin{equation}
\frac{\theta_{-1}}{s}+\theta_{0}+\theta_{1}s=\det\left[
\begin{array}
[c]{cc}%
(\alpha_{1}+s\beta_{1})\omega_{1}^{k_{1}} & (\alpha_{1}+\frac{1}{s}\beta
_{1})\omega_{2}^{k_{1}}\\
(\alpha_{2}+s\beta_{2})\omega_{1}^{k_{2}} & (\alpha_{2}+\frac{1}{s}\beta
_{2})\omega_{2}^{k_{2}}%
\end{array}
\right]  \tag{5}%
\end{equation}
are different from zero, where $\omega_{1}$ and $\omega_{2}$ are the distinct
square roots of $-1$ (see [7] p. 57). The regular boundary conditions are said
to be strongly regular if $\theta_{0}^{2}-4\theta_{1}\theta_{-1}\neq0$. If the
boundary conditions (4) is regular then the equality (5) is either%
\begin{equation}
\frac{\theta_{-1}}{s}+\theta_{0}+\theta_{1}s=\frac{1}{s}-s \tag{6}%
\end{equation}
or
\begin{equation}
\frac{\theta_{-1}}{s}+\theta_{0}+\theta_{1}s=a\left(  s+\frac{1}{s}\right)  +b
\tag{7}%
\end{equation}
(see p. 63 of [7]), where $a,b$ are complex numbers and $a\neq0.$

The eigenvalues of the operator $L_{1}(q)$ generated in the space
$L_{2}\left[  0,1\right]  $ by the differential expression $-y^{^{\prime
\prime}}(x)+q\left(  x\right)  y(x)$ and\ strongly regular boundary conditions
(4), where $q$ is a summable function, consist of the sequences $\{\rho
_{n}^{\left(  1\right)  }(q)\}$ and $\{\rho_{n}^{\left(  2\right)  }%
(q)\}$\ satisfying
\begin{equation}
\rho_{n}^{\left(  1\right)  }(q)=(2n\pi+\gamma_{1})^{2}+O(1),\text{ }\rho
_{n}^{\left(  2\right)  }(q)=(2n\pi+\gamma_{2})^{2}+O(1);\text{ \ }n\geq N>>1,
\tag{8}%
\end{equation}
where
\begin{equation}
\gamma_{j}=-i\ln\zeta_{j},\operatorname{Re}\gamma_{j}\in(-\pi,\pi],\text{
}\zeta_{1}\neq\zeta_{2} \tag{9}%
\end{equation}
and $\zeta_{1},$ $\zeta_{2}$\ are the roots of the equation
\begin{equation}
\theta_{1}\zeta^{2}+\theta_{0}\zeta+\theta_{-1}=0 \tag{10}%
\end{equation}
(see [7], p. 65, formulas (45a), (45b)).

One can readily see that in case (6) we have
\begin{equation}
\gamma_{1}=0,\gamma_{2}=\pi. \tag{11}%
\end{equation}
In case (7), equation (10) has the form
\[
\zeta^{2}+\frac{b}{a}\zeta+1=0,
\]
that is, $\zeta_{1}\zeta_{2}=1$ and by (9) $\zeta_{1}\neq$ $\zeta_{2}$ which
implies that $\zeta_{1}\neq\pm1$ and $\zeta_{2}\neq\pm1.$ Therefore, in case
(7) we have
\begin{equation}
\gamma_{1}=-\gamma_{2}\neq\pi k\text{.} \tag{12}%
\end{equation}

In this paper, first we prove that if the boundary conditions (4) are regular,
then the boundary conditions (3) are also regular (see Theorem 1). Then, as in
[13], we consider the operator $L_{m}(Q)$ as perturbation of $L_{m}(C)$\ by
$Q-C$ and obtain asymptotic formulas for the eigenvalues and eigenfunctions of
$L_{m}(Q)$ in term of the eigenvalues and eigenfunctions of $L_{m}(C),$ where
$C$ is defined in (2). Finally, using the obtained asymptotic formulas and the
theorem of Bari (see p.310 of [2]), we prove that if the eigenvalues of the
matrix $C$ are simple, then the root functions of the operator $L_{m}(Q)$ form
a Riesz basis.

\section{Main Results}

First, let us consider the boundary conditions (3).

\begin{theorem}
If the boundary conditions (4) are regular then the boundary conditions (3)
are also regular.
\end{theorem}

\begin{proof}
The conditions (3) are regular (see [7], p. 121) if the numbers $\Theta_{-m}$
, $\Theta_{m}$\ defined by the identity
\begin{equation}
\Theta_{-m}s^{-m}+\Theta_{-m+1}s^{m-1}+...+\Theta_{m}s^{m}=\det M(m) \tag{13}%
\end{equation}
are both different from zero, where%
\[
M(m)=\left[
\begin{array}
[c]{cc}%
(\alpha_{1}+s\beta_{1})\omega_{1}^{k_{1}}I & (\alpha_{1}+\frac{1}{s}\beta
_{1})\omega_{2}^{k_{1}}I\\
(\alpha_{2}+s\beta_{2})\omega_{1}^{k_{2}}I & (\alpha_{2}+\frac{1}{s}\beta
_{2})\omega_{2}^{k_{2}}I
\end{array}
\right]
\]
and $I$ is $m\times m$ identity matrix. One can easily see that the
intersection of the first and $(m+1)$-th rows and columns forms the matrix
\[
M(1)=\left[
\begin{array}
[c]{cc}%
(\alpha_{1}+s\beta_{1})\omega_{1}^{k_{1}} & (\alpha_{1}+\frac{1}{s}\beta
_{1})\omega_{2}^{k_{1}}\\
(\alpha_{2}+s\beta_{2})\omega_{1}^{k_{2}} & (\alpha_{2}+\frac{1}{s}\beta
_{2})\omega_{2}^{k_{2}}%
\end{array}
\right]
\]
and its complementary minor is $M(m-1).$ Moreover, the determinant of the
minors of $M(m)$ formed by intersection of the first and $(m+1)$-th rows and
other pairs of columns is zero, since the $2\times m$ matrix consisting of
these rows has the form%
\[
\left[
\begin{array}
[c]{cccccccccc}%
(\alpha_{1}+s\beta_{1})\omega_{1}^{k_{1}} & 0 & 0 & \cdot\cdot\cdot & 0 &
(\alpha_{1}+\frac{1}{s}\beta_{1})\omega_{2}^{k_{1}} & 0 & 0 & ... & 0\\
(\alpha_{2}+s\beta_{2})\omega_{1}^{k_{2}} & 0 & 0 & ... & 0 & (\alpha
_{2}+\frac{1}{s}\beta_{2})\omega_{2}^{k_{2}} & 0 & 0 & ... & 0
\end{array}
\right]  .
\]
Therefore using the Laplace's cofactor expansion along the first and
$(m+1)$-th rows we obtain%
\begin{equation}
\det M(m)=\det M(1)\det M(m-1). \tag{14}%
\end{equation}
By induction the formula (14) implies that%
\begin{equation}
\det M(m)=\left(  \det M(1)\right)  ^{m}. \tag{15}%
\end{equation}
Now it follows from (5),\ (6) and (7) that%
\[
\Theta_{m}=\left(  \theta_{1}\right)  ^{m}\text{, }\Theta_{-m}=\left(
\theta_{-1}\right)  ^{m}%
\]
which implies that the boundary conditions (3) are regular if (4) are regular.
\end{proof}

By (10), (13) and (15), $\zeta_{1}$ and $\zeta_{2}$\ are the roots of the
equation
\[
\Theta_{-m}\zeta^{-m}+\Theta_{-m+1}\zeta^{m-1}+...+\Theta_{m}\zeta^{m}=0
\]
with multiplicity $m.$ Therefore it follows from Theorem 2 in p.123 of [7]
that to each root $\zeta_{1}$ and $\zeta_{2}$ correspond $m$ sequences,
denoted by
\[
\{\lambda_{k,1}^{(1)}:k=N,N+1,...\mathbb{\}},\{\lambda_{k,2}^{(1)}%
:k=N,N+1,...\mathbb{\}},...,\{\lambda_{k,m}^{(1)}:k=N,N+1,...\mathbb{\}}%
\]
and%
\[
\{\lambda_{k,1}^{(2)}:k=N,N+1,...\mathbb{\}},\{\lambda_{k,2}^{(2)}%
:k=N,N+1,...\mathbb{\}},...,\{\lambda_{k,m}^{(2)}:k=N,N+1,...\mathbb{\}}%
\]
respectively, satisfying%
\begin{equation}
\lambda_{k,j}^{(1)}=(2k\pi+\gamma_{1})^{2}+O(k^{1-\frac{1}{m}})\text{,
}\lambda_{k,j}^{(2)}=(2k\pi+\gamma_{2})^{2}+O(k^{1-\frac{1}{m}}) \tag{16}%
\end{equation}
for $k=N,N+1,...$ and $j=1,2,...,m,$ where $N\gg1.$

Now to analyze the operators $L_{m}(0),$ $L_{m}(C)$ and $L_{m}(Q),$ we
introduce the following notations. To simplify the notations we omit the upper
indices in $\rho_{n}^{\left(  1\right)  }(0),$ $\rho_{n}^{\left(  2\right)
}(0),$ $\lambda_{k,j}^{(1)},$ $\lambda_{k,j}^{(2)}$ (see (8) and (16)) and
enumerate these eigenvalues in the following way%
\begin{equation}
\rho_{n}^{\left(  1\right)  }(0)=:\rho_{n},\text{ }\rho_{n}^{\left(  2\right)
}(0)=:\rho_{-n},\text{ }\lambda_{k,j}^{(1)}=:\lambda_{k,j},\lambda_{k,j}%
^{(2)}=:\lambda_{-k,j} \tag{17}%
\end{equation}
for $n>0$ and $k\geq N\gg1.$ We remark that there is one-to-one correspondence
between the eigenvalues (counting with multiplicities) of the operator
$L_{1}(0)$ and integers which preserve asymptotic (8). This statement can
easily be proved in a standard way by using Rouche's theorem (we omit the
proof of this fact, since it is used only to simplify the notations). Denote
the normalized eigenfunction of the operator $L_{1}(0)$ corresponding to the
eigenvalue $\rho_{n}$ by $\varphi_{n}$. Clearly,%
\begin{equation}
\varphi_{n,1}=(\varphi_{n},0,0,...0)^{T},\text{ }\varphi_{n,2}=(0,\varphi
_{n},0,...0)^{T},...,\text{ }\varphi_{n,m}=(0,0,...0,\varphi_{n})^{T} \tag{18}%
\end{equation}
are the eigenfunctions of the operator $L_{m}(0)$ corresponding to the
eigenvalue $\rho_{n}$. Similarly,%
\begin{equation}
\varphi_{n,1}^{\ast}=(\varphi_{n}^{\ast},0,0,...0)^{T},\text{ }\varphi
_{n,2}^{\ast}=(0,\varphi_{n}^{\ast},0,...0)^{T},...,\text{ }\varphi
_{n,m}^{\ast}=(0,0,...0,\varphi_{n}^{\ast})^{T} \tag{19}%
\end{equation}
are the eigenfunctions of the operator $L_{m}^{\ast}(0)$ corresponding to the
eigenvalue $\overline{\rho_{n}}$, where $\varphi_{n}^{\ast}$ is the
eigenfunction of $L_{1}^{\ast}(0)$ corresponding to the eigenvalue
$\overline{\rho_{n}}.$

Since the boundary conditions (4) are strongly regular, all eigenvalues of
sufficiently large modulus\ of $L_{1}(q)$ are simple (see the end of \ Theorem
2 of p.65, [7]). Therefore, there exists $n_{0}$ such that the eigenvalues
$\rho_{n}$ of $L_{1}(0)$ are simple for $\left\vert n\right\vert >n_{0}.$
However, the operator $L_{1}(0)$ may have associated functions $\varphi
_{n}^{\left(  1\right)  },$ $\varphi_{n}^{\left(  2\right)  },...,\varphi
_{n}^{\left(  t(n)\right)  }$ corresponding to the eigenfunction $\varphi_{n}$
for $\left\vert n\right\vert \leq n_{0}.$ Then, it is not hard to see that
$L_{m}(0)$ has associated functions%
\[
\varphi_{n,1,p}=(\varphi_{n}^{\left(  p\right)  },0,0,...0)^{T},\varphi
_{n,2,p}=(0,\varphi_{n}^{\left(  p\right)  },0,...0)^{T},...,\varphi
_{n,m,p}=(0,0,...0,\varphi_{n}^{\left(  p\right)  })^{T},
\]
for $p=1,2,...,t(n)$ corresponding to $\rho_{n}$ for $\left\vert n\right\vert
\leq n_{0},$ that is,%
\begin{align*}
(L_{m}(0)-\rho_{n})\varphi_{n,i,0}  &  =0,\\
(L_{m}(0)-\rho_{n})\varphi_{n,i,p}  &  =\varphi_{n,i,p-1},\text{
\ }p=1,2,...,t(n),
\end{align*}
where $\varphi_{n,i,0}(x)$ $=:\varphi_{n,i}(x).$ Since the system of the root
functions of $L_{1}(0)$ forms Riesz basis in $L_{2}\left(  0,1\right)  $ (see
[6]), the system%
\begin{equation}
\{\varphi_{n,i,p}:n\in\mathbb{Z},\text{ }i=1,2,...,m,\text{ }p=1,2,...,t(n)\}
\tag{20}%
\end{equation}
forms a Riesz basis in $L_{2}^{m}\left(  0,1\right)  .$ The system,%
\begin{equation}
\{\varphi_{n,i,p}^{\ast}:n\in\mathbb{Z},\text{ }i=1,2,...,m,\text{
}p=1,2,...,t(n)\} \tag{21}%
\end{equation}
which is biorthogonal to $\{\varphi_{n,i,p}\}$ is the system of the
eigenfunctions and the associated functions of the adjoint operator
$L_{m}^{\ast}(0).$ Clearly, (21) can be constructed by repeating the
construction of (20)\ and replacing everywhere $\varphi_{n}$ by $\varphi
_{n}^{\ast}.$ Thus%
\begin{align}
(L_{m}^{\ast}(0)-\overline{\rho_{n}})\varphi_{n,i,0}^{\ast}  &  =0,\tag{22}\\
(L_{m}^{\ast}(0)-\overline{\rho_{n}})\varphi_{n,i,p}^{\ast}  &  =\varphi
_{n,i,p-1}^{\ast},\text{ \ }p=1,2,...,t(n). \tag{23}%
\end{align}

To prove the main results, we need the following properties of the
eigenfunctions $\varphi_{n}$ and $\varphi_{n}^{\ast}.$

\begin{proposition}
If the boundary conditions (4) are strongly regular then there exists a
positive constant $M$ such that%
\begin{equation}
\sup_{x\in\lbrack0,1]}\left\vert \varphi_{n}(x)\right\vert \leq M,\text{ }%
\sup_{x\in\lbrack0,1]}\left\vert \varphi_{n}^{\ast}(x)\right\vert \leq M,
\tag{24}%
\end{equation}%
\begin{equation}
\sup_{x\in\lbrack0,1]}\left\vert \varphi_{n,i,p}(x)\right\vert \leq M,\text{
}\sup_{x\in\lbrack0,1]}\left\vert \varphi_{n,i,p}^{\ast}(x)\right\vert \leq M,
\tag{25}%
\end{equation}
for all $n,i,p,$ where $\varphi_{n}^{\ast}$ is the eigenfunctions of
$L_{1}^{\ast}(0),$ satisfying%
\begin{equation}
\left(  \varphi_{n},\varphi_{n}^{\ast}\right)  =1 \tag{26}%
\end{equation}
for $\left\vert n\right\vert >n_{0}.$ Moreover, the following asymptotic
formulas hold%
\begin{align}
\overline{\varphi_{n}^{\ast}(x)}\varphi_{n}(x)  &  =1+A_{1}e^{i(4\pi
n+2\gamma_{1})x}+B_{1}e^{-i(4\pi n+2\gamma_{1})x}+O(\frac{1}{n}),n>0,\tag{27}%
\\
\overline{\varphi_{n}^{\ast}(x)}\varphi_{n}(x)  &  =1+A_{2}e^{i(4\pi
n+2\gamma_{2})x}+B_{2}e^{-i(4\pi n+2\gamma_{2})x}+O(\frac{1}{n}),n<0,\nonumber
\end{align}
where $A_{j}$ and $B_{j}$ for $j=1,2$ are constants.
\end{proposition}

\begin{proof}
It is well-known that (see p. 62-63 of [7]) if the boundary conditions (4) are
strongly regular then these are either Dirichlet boundary conditions:
$y(1)=y(0)=0$ or%
\begin{gather}
y^{^{\prime}}(0)+\alpha_{11}y(0)+\alpha_{12}y(1)=0,\tag{28}\\
y^{^{\prime}}(1)+\alpha_{21}y(0)+\alpha_{22}y(1)=0\nonumber
\end{gather}
or%
\begin{align}
a_{1}y^{^{\prime}}(0)+b_{1}y^{^{\prime}}(1)+a_{0}y(0)+b_{0}y(1)  &
=0,\tag{29}\\
c_{0}y(0)+d_{0}y(1)  &  =0.\nonumber
\end{align}
In the case of Dirichlet boundary conditions we have%
\begin{equation}
\varphi_{n}(x)=\varphi_{n}^{\ast}(x)=\sqrt{2}\sin2nx,\text{ }\varphi
_{-n}(x)=\varphi_{-n}^{\ast}(x)=\sqrt{2}\sin(2n+1)x, \tag{30}%
\end{equation}
where $n=1,2,...$. In case (28) using the well known expression for
eigenfunction%
\[
\left\vert
\begin{array}
[c]{cc}%
e^{i\rho_{n}x} & e^{-i\rho_{n}x}\\
U_{1}(e^{i\rho_{n}x}) & U_{1}(e^{-i\rho_{n}x})
\end{array}
\right\vert
\]
and (8), (17), and also taking into account that $\gamma_{1}=0$ and
$\gamma_{2}=\pi$ (see (6) and (11)), we obtain%
\begin{equation}
\varphi_{n}(x)=\sqrt{2}\cos2nx+O(\frac{1}{n}),\text{ }\varphi_{-n}(x)=\sqrt
{2}\cos(2n+1)x+O(\frac{1}{n}) \tag{31}%
\end{equation}
and%
\begin{equation}
\varphi_{n}^{\ast}(x)=\sqrt{2}\cos2nx+O(\frac{1}{n}),\text{ }\varphi
_{-n}^{\ast}(x)=\sqrt{2}\cos(2n+1)x+O(\frac{1}{n}). \tag{32}%
\end{equation}
In the same way in case (29) we get the formulas%
\begin{align}
\varphi_{n}(x)  &  =a^{+}e^{i(2\pi n+\gamma_{1})x}+b^{+}e^{-i(2\pi
n+\gamma_{1})x}+O(\frac{1}{n}),\tag{33}\\
\varphi_{n}^{\ast}(x)  &  =c^{+}e^{i(2\pi n+\gamma_{1})x}+d^{+}e^{-i(2\pi
n+\gamma_{1})x}+O(\frac{1}{n})\nonumber
\end{align}
for $n>0$ and%
\begin{align}
\varphi_{n}(x)  &  =a^{-}e^{i(2\pi n+\gamma_{2})x}+b^{-}e^{-i(2\pi
n+\gamma_{2})x}+O(\frac{1}{n}),\tag{34}\\
\varphi_{n}^{\ast}(x)  &  =c^{-}e^{i(2\pi n+\gamma_{2})x}+d^{-}e^{-i(2\pi
n+\gamma_{2})x}+O(\frac{1}{n})\nonumber
\end{align}
for $n<0$. Thus in any case inequality (24) holds. Equality (25) follows from
(24)\ and equality (27) follows from (26), (30)-(34).
\end{proof}

As it is noted in the introduction, we obtain asymptotic formulas for the
eigenvalues and eigenfunctions of $L_{m}(Q)$ in term of the eigenvalues and
eigenfunctions of $L_{m}(C).$ Therefore first we analyze the eigenvalues and
eigenfunctions of $L_{m}(C).$ Suppose that the matrix $C$ has $m$ simple
eigenvalues $\mu_{1},\mu_{2},...,\mu_{m}$. The normalized eigenvector
corresponding to the eigenvalue $\mu_{j}$ is denoted by $v_{j}.$ In these
notations the eigenvalues and eigenfunctions of $L_{m}(C)$ are
\begin{equation}
\mu_{k,j}=\rho_{k}+\mu_{j}\text{ }\And\Phi_{k,j}(x)=v_{j}\varphi_{k}(x)
\tag{35}%
\end{equation}
respectively. Similarly, the eigenvalues and eigenfunctions of $(L_{m}%
(C))^{\ast}$ are $\overline{\mu_{k,j}},$ $\Phi_{k,j}^{\ast}(x)=v_{j}^{\ast
}\varphi_{k}^{\ast},$ where $v_{j}^{\ast}$ is the eigenvector of \ $C^{\ast}$
corresponding to $\overline{\mu_{j}}$ such that $\left(  v_{j}^{\ast}%
,v_{j}\right)  =1.$ To obtain the asymptotic formulas for the eigenvalues and
eigenfunctions of $L_{m}(Q),$ we use the following formula%
\begin{equation}
(\lambda_{k,j}-\mu_{k,i})(\Psi_{k,j},\Phi_{k,i}^{\ast})=((Q-C)\Psi_{k,j}%
,\Phi_{k,i}^{\ast}) \tag{36}%
\end{equation}
obtained from%
\begin{equation}
L_{m}(Q)\Psi_{k,j}(x)=\lambda_{k,j}\Psi_{k,j}(x) \tag{37}%
\end{equation}
by multiplying both sides of (37) with $\Phi_{k,i}^{\ast}(x)$ and using
$L_{m}(Q)=L_{m}(C)+(Q-C)$. To prove that $\lambda_{k,j}$ is close to
$\mu_{k,j},$ we first show that the right-hand side of (36) is a small number
for all $j$ and $i$\ (see Lemma 1) and then we prove that for each
eigenfunction $\Psi_{k,j}$ of $L_{m}(Q)$, where $|k|\geq N,$ there exists a
root function of $(L_{m}(C)))^{\ast}$ denoted by $\Phi_{k,j}^{\ast}$ such that
$(\Psi_{k,j},\Phi_{k,j}^{\ast})$ is a number of order $1$ (see Lemma 2).
\ Before the proof of these lemmas, we need the following preparations:
Multiplying both sides of (37) by $\varphi_{n,i,0}^{\ast}$, using
$L_{m}(Q)=L_{m}(0)+Q$ and (22) we get%
\[
(\lambda_{k,j}-\rho_{n})(\Psi_{k,j},\varphi_{n,i,0}^{\ast})=(Q\Psi
_{k,j},\varphi_{n,i,0}^{\ast}),
\]%
\begin{equation}
(\Psi_{k,j},\varphi_{n,i,0}^{\ast})=\frac{(Q\Psi_{k,j},\varphi_{n,i,0}^{\ast
})}{\lambda_{k,j}-\rho_{n}} \tag{38}%
\end{equation}
for\ $i,j=1,2,...,m$ and $\lambda_{k,j}\neq\rho_{n}.$ Now multiplying (37) by
$\varphi_{n,i,1}^{\ast}$ and using (23), (38), we get
\[
(\lambda_{k,j}-\rho_{n})(\Psi_{k,j},\varphi_{n,i,1}^{\ast})=(Q\Psi
_{k,j},\varphi_{n,i,1}^{\ast})+\frac{(Q\Psi_{k,j},\varphi_{n,i,0}^{\ast}%
)}{\lambda_{k,j}-\rho_{n}},
\]%
\[
(\Psi_{k,j},\varphi_{n,i,1}^{\ast})=\frac{(Q\Psi_{k,j},\varphi_{n,i,1}^{\ast
})}{\lambda_{k,j}-\rho_{n}}+\frac{(Q\Psi_{k,j},\varphi_{n,i,0}^{\ast}%
)}{(\lambda_{k,j}-\rho_{n})^{2}}.
\]
In this way one can deduce the formulas
\begin{equation}
(\Psi_{k,j},\varphi_{n,i,s}^{\ast})=\sum_{p=0}^{s}\frac{(Q\Psi_{k,j}%
,\varphi_{n,i,p}^{\ast})}{(\lambda_{k,j}-\rho_{n})^{s+1-p}} \tag{39}%
\end{equation}
for $s=0,1,...,t(n).$

Since, $(Q\Psi_{k,j},\varphi_{n,i,p}^{\ast})=(\Psi_{k,j},Q^{\ast}%
\varphi_{n,i,p}^{\ast})$, $\left\Vert \Psi_{k,j}\right\Vert =1,$ and the
entries of $Q$ are the elements of $L_{2}\left(  0,1\right)  ,$ it follows
from (25) and Cauchy-Schwarz inequality that there exists a positive constant
$c_{1}$ such that
\begin{equation}
\left\vert (Q\Psi_{k,j},\varphi_{n,i,p}^{\ast})\right\vert <c_{1}. \tag{40}%
\end{equation}
In the subsequent estimates we denote by $c_{m}$ for $m=1,2,...,$ the positive
constants whose exact value are inessential. On the other hand, it follows
from (8), (16) and (17) that if $k$ is a sufficiently large number then%
\begin{equation}
\mid\lambda_{k,j}-\rho_{p}\mid>c_{2}k^{2},\forall p\leq n_{0}, \tag{41}%
\end{equation}%
\begin{equation}
\left\vert \lambda_{k,j}-\rho_{p}\right\vert >c_{3}\left\vert \left\vert
p\right\vert -\left\vert k\right\vert \right\vert (\left\vert p\right\vert
+\left\vert k\right\vert ),\forall p\neq\pm k, \tag{42}%
\end{equation}
and%
\begin{equation}
\left\vert \lambda_{k,j}-\rho_{-k}\right\vert >c_{4}\left\vert k\right\vert .
\tag{43}%
\end{equation}
Therefore by (38)-(43) we have%
\begin{equation}
\left\vert (\Psi_{k,j},\varphi_{p,q,s}^{\ast})\right\vert <c_{5}k^{-2},\forall
p\leq n_{0}, \tag{44}%
\end{equation}%
\begin{equation}
\mid(\Psi_{k,j},\varphi_{p,q}^{\ast})\mid\leq\frac{c_{6}}{\left\vert
\left\vert p\right\vert -\left\vert k\right\vert \right\vert (\left\vert
p\right\vert +\left\vert k\right\vert )},\text{ }\mid(\Psi_{k,j}%
,\varphi_{-k,q}^{\ast})\mid\leq\frac{c_{7}}{\left\vert k\right\vert } \tag{45}%
\end{equation}
for all $\left\vert k\right\vert \gg1,$ $p\neq\pm k,$ $s=0,1,...,t(p)$ and
$q,j=1,2,...,m.$

Now, we are ready to prove the lemmas.

\begin{lemma}
For any $i=1,2,...,m$ and $j=1,2,...,m$ the following estimation holds
\begin{equation}
\left(  \Psi_{k,j},(Q^{\ast}-C^{\ast})\Phi_{k,i}^{\ast}\right)  =O(\alpha
_{k})+O(\frac{\ln|k|}{k}), \tag{46}%
\end{equation}
where $\alpha_{k}=\max\left\{  \mid b_{s,i,2k,r}^{+}\mid,\mid b_{s,i,2k,r}%
^{-}\mid:s,i=1,2,...,m;\text{ }r=1,2\right\}  ,$%
\begin{equation}
b_{s,i,2k,r}^{\pm}=\int_{0}^{1}b_{s,i}(x)e^{\pm i(4\pi k+2\gamma_{r})x},
\tag{47}%
\end{equation}
$b_{s,i}\in L_{2}\left(  0,1\right)  $ are the entries of the matrix $Q,$ and
$\gamma_{r}$ is defined in (9).
\end{lemma}

\begin{proof}
Since $\Phi_{k,i}^{\ast}(x)=v_{i}^{\ast}\varphi_{k}^{\ast}(x),$ it is enough
to prove that%
\begin{equation}
\left(  \Psi_{k,j},(Q^{\ast}-C^{\ast})\varphi_{k,s}^{\ast}\right)  =O\left(
\alpha_{k}\right)  +O(\frac{\ln|k|}{k}) \tag{48}%
\end{equation}
for $s=1,2,...,m.$ The decomposition of $(Q^{\ast}-C^{\ast})\varphi
_{k,s}^{\ast}$ by the basis (21) has the form%
\[
(Q^{\ast}-C^{\ast})\varphi_{k,s}^{\ast}=\sum\limits_{q=1}^{m}\sum
_{p:\left\vert p\right\vert \leq n_{0}}\sum\limits_{v=1}^{t(p)}%
c(k,s,p,q,v)\varphi_{p,q,v}^{\ast}+
\]%
\[
\sum\limits_{q=1,2,...m}\sum_{p:\left\vert p\right\vert >n_{0}}^{\infty
}((Q^{\ast}-C^{\ast})\varphi_{k,s}^{\ast},\varphi_{p,q})\varphi_{p,q}^{\ast}%
\]
Therefore%
\begin{align}
(\Psi_{k,j},(Q^{\ast}-C^{\ast})\varphi_{k,s}^{\ast})  &  =\sum\limits_{q=1}%
^{m}\sum_{p:\left\vert p\right\vert \leq n_{0}}\sum\limits_{v=1}%
^{t(p)}\overline{c(k,s,p,q,v)}(\Psi_{k,j},\varphi_{p,q,v}^{\ast})+\tag{49}\\
&  \sum_{q=1}^{m}\sum_{p:\left\vert p\right\vert >n_{0}}\overline{((Q^{\ast
}-C^{\ast})\varphi_{k,s}^{\ast},\varphi_{p,q})}(\Psi_{k,j},\varphi_{p,q}%
^{\ast})\nonumber
\end{align}
Since $c(k,s,p,q,v)=O(1),$ by (44) the first summation of the right hand side
of (49) is $O(k^{-2}).$

Now let us estimate the second summation $S$ of the right hand side of (49).
It can be written in the form
\begin{equation}
S=S_{1}+S_{2}, \tag{50}%
\end{equation}
where%
\begin{align*}
S_{1}  &  =\sum_{q=1}^{m}\overline{((Q^{\ast}-C^{\ast})\varphi_{k,s}^{\ast
},\varphi_{k,q})}(\Psi_{k,j},\varphi_{k,q}^{\ast}),\\
S_{2}  &  =\sum_{q=1}^{m}\sum_{p\neq k}\overline{((Q^{\ast}-C^{\ast}%
)\varphi_{k,s}^{\ast},\varphi_{p,q})}(\Psi_{k,j},\varphi_{p,q}^{\ast}).
\end{align*}
Using (18), (19), (25) and (27), one can easily verify that%
\begin{equation}
S_{1}=O\left(  \alpha_{k}\right)  +O(\frac{1}{k}) \tag{51}%
\end{equation}
On the other hand, by (45), we have%
\begin{equation}
S_{2}=O(\frac{\ln|k|}{k}) \tag{52}%
\end{equation}
Therefore, (48) follows from (49)-(52). The lemma is proved.
\end{proof}

\begin{lemma}
For each eigenfunction $\Psi_{k,j}$ of $L_{m}(Q)$, where $|k|\geq N,$ there
exists an eigenfunction of $(L_{m}(C))^{\ast}$ denoted by $\Phi_{k,j}^{\ast}$
such that%
\begin{equation}
\left\vert \left(  \Psi_{k,j},\Phi_{k,j}^{\ast}\right)  \right\vert >c_{8}.
\tag{53}%
\end{equation}

\end{lemma}

\begin{proof}
Since (20) is a Riesz basis of $L_{2}^{m}\left[  0,1\right]  ,$ we have%
\begin{equation}
\Psi_{k,j}=\sum\limits_{q=1}^{m}\sum_{p:\left\vert p\right\vert \leq n_{0}%
}\sum\limits_{v=1}^{t(p)}c(k,j,p,q,v)\varphi_{p,q,v}+\sum\limits_{q=1}^{m}%
\sum_{p:\left\vert p\right\vert >n_{0}}\left(  \Psi_{k,j},\varphi_{p,q}^{\ast
}\right)  \varphi_{p,q}. \tag{54}%
\end{equation}
It follows from (25) and (45) that%
\begin{equation}
\sum\limits_{q=1}^{m}\sum_{\substack{p\neq k\\\left\vert p\right\vert >n_{0}%
}}\left\Vert \left(  \Psi_{k,j},\varphi_{p,q}^{\ast}\right)  \varphi
_{p,q}\right\Vert =O(\frac{\ln\left\vert k\right\vert }{k}). \tag{55}%
\end{equation}
On the other hand, arguing as in the estimation for the first summation of
(49) we get%
\[
\sum\limits_{q=1}^{m}\sum_{p:\left\vert p\right\vert \leq n_{0}}%
\sum\limits_{j=1}^{t(p)}\left\Vert c(k,s,p,q,j)\varphi_{p,q,j}\right\Vert
=O(\frac{1}{k^{2}}).
\]
Therefore using (55), (54), we obtain%
\begin{equation}
\Psi_{k,j}=\sum\limits_{q=1}^{m}\left(  \Psi_{k,j},\varphi_{k,q}^{\ast
}\right)  \varphi_{k,q}+O(\frac{\ln\left\vert k\right\vert }{k}) \tag{56}%
\end{equation}
Since $\left\{  \varphi_{k,1},\varphi_{k,2},....,\varphi_{k,m}\right\}  $ is
orthonormal system and $\left\Vert \Psi_{k,j}\right\Vert =1$, there exists an
index $q$ such that
\begin{equation}
\left\vert \left(  \Psi_{k,j},\varphi_{k,q}^{\ast}\right)  \right\vert >c_{9}.
\tag{57}%
\end{equation}
On the other hand%
\begin{equation}
\varphi_{k,q}^{\ast}=\sum\limits_{j=1}^{m}\left(  \varphi_{k,q}^{\ast}%
,\Phi_{k,j}\right)  \Phi_{k,j}^{\ast} \tag{58}%
\end{equation}
because $\Phi_{k,j}^{\ast}=v_{j}^{\ast}\varphi_{k}^{\ast}$, and the vectors
$v_{j}^{\ast},$ $j=1,2,...m$ form a basis in $\mathbb{C}^{m}.$ Now, using (58)
in (57), we get the proof of the lemma.
\end{proof}

\begin{theorem}
Suppose that all eigenvalues $\mu_{1},\mu_{2},...,\mu_{m}$ of the matrix $C$
are simple. Then, there exists a number $N$ such that all eigenvalues
$\lambda_{k,1},\lambda_{k,2},...,\lambda_{k,m}$ of $L_{m}\left(  Q\right)  $
for $\mid k\mid\geq N$ are simple and satisfy the asymptotic formula
\begin{equation}
\lambda_{k,j}=\mu_{k,j}+O(\alpha_{k})+O(\frac{\ln\left\vert k\right\vert }%
{k}), \tag{59}%
\end{equation}
where $\mu_{k,j}$ is the eigenvalue of $L_{m}\left(  C\right)  $ and
$\alpha_{k}$ is defined in Lemma 1. The normalized eigenfunction $\Psi
_{k,j}(x)$ of $L_{m}(Q)$ corresponding to $\lambda_{k,j}$ satisfies%
\begin{equation}
\Psi_{k,j}(x)=\Phi_{k,j}(x)+O(\alpha_{k})+O(\frac{\ln\left\vert k\right\vert
}{k}), \tag{60}%
\end{equation}
where $\Phi_{k,j}(x)$ is the normalized eigenfunction of $L_{m}\left(
C\right)  $ corresponding to $\mu_{k,j}.$ The root functions of $L_{m}\left(
Q\right)  $ form a \ Riesz basis in $L_{2}^{m}(0,1).$
\end{theorem}

\begin{proof}
In (36) replacing $i$ by $j,$ and then dividing the both sides of the obtained
equality by $(\Psi_{k,j},\Phi_{k,j}^{\ast})$ and using Lemma 1 and Lemma 2, we
see all large eigenvalues of $L_{m}\left(  Q\right)  $ lie in $r_{k}$
neighborhood of the eigenvalues $\mu_{k,j}$ for $\mid k\mid\geq N,$
$j=1,2,...,m$ of $L_{m}(C),$ where%
\begin{equation}
r_{k}=O(\alpha_{k})+O(\frac{\ln\left\vert k\right\vert }{k}). \tag{61}%
\end{equation}

Now we prove that these eigenvalues are simple. Let $\lambda_{k,j}$ be an
eigenvalue of $L_{m}\left(  Q\right)  $ lying in $\frac{1}{2}a_{j}$
neighborhood of $\mu_{k,j}=\rho_{k}+\mu_{j}$ (see (35)), where $a_{j}%
=\min_{i\neq j}\mid\mu_{j}-\mu_{i}\mid.$ Then, by triangle inequality%
\[
\mid\lambda_{k,j}-\mu_{k,i}\mid>\mid\mu_{k,j}-\mu_{k,i}\mid-\mid\lambda
_{k,j}-\mu_{k,j}\mid\geq a_{j}-\frac{1}{2}a_{j}=\frac{1}{2}a_{j}%
\]
for $i\neq j.$ Therefore using (36) and Lemma 1 we get%
\[
(\Psi_{k,j},\Phi_{k,i}^{\ast})=O\left(  \alpha_{k}\right)  +O(\frac
{\ln\left\vert k\right\vert }{k})
\]
for $i\neq j$. This and (56) imply that (60) holds for any normalized
eigenfunction corresponding to $\lambda_{k,j},$ since%
\[
span\left\{  \varphi_{k,1},\varphi_{k,2},...,\varphi_{k,m}\right\}
=span\{\Phi_{k,1,},\Phi_{k,2,},...,\Phi_{k,m,}\}.
\]
Using this, let us prove that $\lambda_{k,j}$ is a simple eigenvalue. Suppose
to the contrary that $\lambda_{k,j}$ is a multiple eigenvalue. If there are
two linearly independent eigenfunctions corresponding to $\lambda_{k,j},$ then
one can find two orthogonal eigenfunctions satisfying (60), which is
impossible. Hence there exists a unique eigenfunction $\Psi_{k,j}$
corresponding to $\lambda_{k,j}.$ If there exists an associated function
$\Psi_{k,j,1}$ belonging to the eigenfunction $\Psi_{k,j},$ then%
\[
(L_{m}(Q)-\lambda_{k,j})\Psi_{k,j,1}(x)=\Psi_{k,j}(x).
\]
Multiplying both sides of this equality by $\Psi_{k,j}^{\ast}(x),$ where
$\Psi_{k,j}^{\ast}(x)$ is the normalized eigenfunction of $(L_{m}(Q))^{\ast}$
corresponding to the eigenvalue $\overline{\lambda_{k,j}},$we obtain%
\begin{equation}
(\Psi_{k,j},\Psi_{k,j}^{\ast})=(\Psi_{k,i,1},((L_{m}(Q))^{\ast}-\overline
{\lambda_{k,j}}))\Psi_{k,j}^{\ast})=0. \tag{62}%
\end{equation}
Since the proved statements are also applicable for the adjoint operator
$(L_{m}(Q))^{\ast},$ formula (60) holds for this operator too, that is, we
have%
\begin{equation}
\Psi_{k,j}^{\ast}(x)=\Phi_{k,j}^{\ast}(x)+O\left(  \alpha_{k}\right)
+O(\frac{\ln\left\vert k\right\vert }{k}). \tag{63}%
\end{equation}
This formula, (60) and the obvious relation $(\Phi_{k,j},\Phi_{k,j}^{\ast})=1$
contradict with (62). Thus, $\lambda_{k,j}$ is a simple eigenvalue.

We proved that all large eigenvalues of $L_{m}\left(  Q\right)  $ lie in the
disk%
\[
\Delta_{k,j}=\left\{  z:\left\vert z-\mu_{k,j}\right\vert <r_{k}\right\}
\]
for $\mid k\mid\geq N$, $j=1,2,...,m,$ where $r_{k}$ is defined in (61).
Clearly, the disks $\Delta_{k,j}$ for $j=1,2,...,m$ and $\mid k\mid\geq N$ are
pairwise disjoint. Let us prove that each of these disks does not contain more
than one eigenvalue of \ $L_{m}\left(  Q\right)  .$ Suppose to the contrary
that, two different eigenvalues $\Lambda_{1}$ and $\Lambda_{2}$ lie in
$\Delta_{k,j}.$ Then it has already been proven that these eigenvalues are
simple and the corresponding eigenfunctions $\Psi_{1}$ and $\Psi_{2}$ satisfy%
\[
\Psi_{p}(x)=\Phi_{k,j}(x)+O\left(  \alpha_{k}\right)  +O(\frac{\ln\left\vert
k\right\vert }{k})
\]
for $p=1,2.$ Similarly, the eigenfunctions $\Psi_{1}^{\ast}$ and $\Psi
_{2}^{\ast}$ of $(L_{m}\left(  Q\right)  )^{\ast}$ corresponding to the
eigenvalues $\overline{\Lambda_{1}}$ and $\overline{\Lambda_{2}}$ satisfy%
\[
\Psi_{p}^{\ast}(x)=\Phi_{k,j}^{\ast}(x)+O\left(  \alpha_{k}\right)
+O(\frac{\ln\left\vert k\right\vert }{k}).
\]
for $p=1,2.$ Since $\Lambda_{1}\neq\Lambda_{2},$ we have%
\[
0=(\Psi_{1},\Psi_{2}^{\ast})=1+O\left(  \alpha_{k}\right)  +O(\frac
{\ln\left\vert k\right\vert }{k})
\]
which is impossible. Hence the pairwise disjoint disks $\Delta_{k,1},$
$\Delta_{k,2},...,$ $\Delta_{k,m}$ , where $|k|\geq N,$ contain $m$
eigenvalues of $L_{m}(Q)$ and each of these disks does not contain more than
one eigenvalue. Therefore, there exists a unique eigenvalue $\lambda_{k,j}$ of
$L_{m}(Q)$ lying in $\Delta_{k,j},$ where $j=1,2,...,m$ and $|k|\geq N.$ Thus
the eigenvalues $\lambda_{k,j}$ for $|k|\geq N$ are simple and the formulas
(59) and (60) hold.

It remains to prove that the root functions of $L_{m}\left(  Q\right)  $ form
a\ Riesz basis in $L_{2}^{m}(0,1).$ For this, let us prove that for $f\in
L_{2}^{m}(0,1)$, the following series is convergent%
\begin{equation}
\sum_{j=1}^{m}\sum\limits_{k=N+1}^{\infty}\left\vert (f,\Psi_{k,j})\right\vert
^{2}<\infty, \tag{64}%
\end{equation}
where $N$ is a large positive number. By the asymptotic formula (60), we have%
\begin{equation}
\sum\limits_{k=N+1}^{\infty}\left\vert (f,\Psi_{k,j})\right\vert ^{2}%
\leq3(\sum\limits_{k=N+1}^{\infty}\left\vert (f,\Phi_{k,j})\right\vert
^{2}+\sum\limits_{k=N+1}^{\infty}\left\vert (f,g_{k})\right\vert ^{2}%
+\sum\limits_{k=N+1}^{\infty}\left\vert (f,h_{k})\right\vert ^{2}) \tag{65}%
\end{equation}
where $\left\Vert g_{k}\right\Vert =O\left(  \alpha_{k}\right)  $ and
$\left\Vert h_{k}\right\Vert =O(\frac{\ln\left\vert k\right\vert }{k}).$ The
first series in the right side of (65) converges, since the root functions of
$L_{m}\left(  C\right)  $ is a Riesz basis in $L_{2}^{m}(0,1).$ Using the
Cauchy-Schwarz inequality we get%
\begin{equation}
\sum\limits_{k=N+1}^{\infty}\left\vert (f,g_{k})\right\vert ^{2}\leq
c_{10}\left\Vert f\right\Vert ^{2}\sum\limits_{k=N+1}^{\infty}\left\vert
\alpha_{k}\right\vert ^{2}. \tag{66}%
\end{equation}
On the other hand, using the definition of $\alpha_{k}$ (see Lemma 1) and
taking into account that the entries of the matrix $Q$ are the element of
$L_{2}\left(  0,1\right)  ,$ we obtain%
\[
\sum\limits_{k=N+1}^{\infty}\left\vert \alpha_{k}\right\vert ^{2}<\infty.
\]
Therefore, by (66), the second series in the right side of (65) converges too.
In the same way, we prove that the third series in the right side of (65)
converges. Thus, the series of the left-hand side of (65) converges, that is,
(64) is proved. By Bari's definition (see [2], chap. 6), this implies that the
system of eigenfunctions of the operator under consideration is Bessel. Since
the system of root functions of the adjoint operator has the asymptotics (63),
in the same way, we obtain that it is also Bessel. Moreover, the equality%
\[
(\Psi_{k,j},\Psi_{k,j}^{\ast})=1+O\left(  \alpha_{k}\right)  +O(\frac
{\ln\left\vert k\right\vert }{k})
\]
(see (60) and (63)) implies that the system of the root functions of $\left(
L_{m}(Q)\right)  ^{\ast},$ which is biorthogonal to the system of the root
functions of $L_{m}(Q),$ is also Bessel. As it is noted in [11,12], the system
of root functions of the operators $L_{m}(Q)$ and $\left(  L_{m}(Q)\right)
^{\ast}$ are complete in the space $L_{2}^{m}(0,1)$. These arguments and
Bari's theorem in p.310 of [2] (if two biorthogonal systems are complete and
Bessel, then they both are Riesz bases) conclude the proof of the theorem.
\end{proof}

\bigskip

\bigskip
\end{document}